\documentclass[11pt,english]{article}

\usepackage[utf8x]{inputenc}
\usepackage[T1]{fontenc}
\usepackage{cmlgc}
\usepackage{ucs}

\usepackage{xcolor}

\usepackage{graphicx}
\usepackage{amsmath}
\usepackage{amsfonts,amssymb}
\usepackage{mathrsfs} 
\usepackage{xcolor}
\usepackage{dsfont}

\definecolor{vertfonce}{rgb}{0.20, 0.46, 0.25}
\definecolor{rougefonce}{rgb}{0.64, 0.09, 0.20}

\usepackage{babel}
\usepackage[breaklinks=true,
            colorlinks=true,
            linkcolor=rougefonce,
            citecolor=vertfonce]{hyperref}

\usepackage{autonum} 
\input{macro}

\newcommand{\moy}[1]{\langle #1 \rangle}
\newcommand{\DD}{\:\!\mathrm{d}}

\title{Uniform spectral asymptotics for semiclassical wells on phase
  space loops}

\author{\textsc{Deleporte} Alix\footnote{Institute for Mathematics,
    University of Z\"urich, Winterthurerstr. 190, CH-8057 Z\"urich} \and \textsc{Vũ Ngọc} 
  San\footnote{Univ Rennes, CNRS, IRMAR - UMR 6625, F-35000 Rennes,
    France}}

\begin{document}
\maketitle


\begin{abstract}
  We consider semiclassical self-adjoint operators whose symbol,
  defined on a two-dimensional symplectic manifold, reaches a
  non-degenerate minimum $b_0$ on a closed curve. We derive a classical and
  quantum normal form which allows us, in addition to the complete
  integrability of the system, to obtain eigenvalue asymptotics in a
  window $(-\infty,b_0+\epsilon)$ for $\epsilon>0$ independent on the
  semiclassical parameter. These asymptotics are obtained in two
  complementary settings: either a symmetry of the system under
  translation along the curve, or a Morse hypothesis reminiscent of
  Helffer-Sjöstrand's ``miniwell'' situation.
\end{abstract}

\section{Wells on closed loops}

Let $(M,\omega)$ be a symplectic surface without boundary. When
introducing quantization, we will assume for simplicity that
$M=T^*\RM$ or $M=T^*S^1$.  Let $\gamma\subset M$ be a smooth embedded
closed loop.  We say that a smooth function $p\in\Cinf(M)$ admits a
non-degenerate well on the loop $\gamma$ if there exists a
neighborhood $\Omega$ of $\gamma$ in $M$ such that
\begin{enumerate}
\item $p_{\restr \Omega}$ is minimal on $\gamma$:
  \begin{equation}
    \label{equ:well}
    p^{-1}(b_0) \cap \Omega=\gamma, \quad \text{ where } 
    \inf_\Omega p = \min_\Omega p = b_0;
  \end{equation}
\item and this minimum is Morse-Bott non-degenerate: at each point
  $m\in\gamma$, the restriction of the Hessian $p''(m)$ to the
  transversal direction to $\gamma$ does not vanish.
\end{enumerate}
Thus by the Morse-Bott lemma, there exists a neighborhood
$\tilde\Omega\subset\Omega$ of $\gamma$, and coordinates
$(x,t): \tilde\Omega \to \gamma\times [-\delta,\delta]$ such that
$\gamma = \{t=0\}$ and $p = b_0 + t^2q(x)$, for some smooth,
non-vanishing function $q$ on $\gamma$.

An example of such a $p$ can be obtained in the following way. Let
$f:M\to\RM$ be smooth and proper, and let $c\in\RM$ be a regular value
of $f$. Let $\gamma$ be a connected component of $f^{-1}(c)$.  Then
define $p=(f-c)^2$. We see that $\gamma$ is a non-degenerate well for
$p$.

As we will see below, this is actually the universal form of a
non-degenerate well. However, this normal form is not sufficient to
describe the semiclassical quantization of our setting, because
assumption~\eqref{equ:well} is \emph{not} stable under
perturbation. In fact, generic perturbations of $p$ create isolated
local extrema along $\gamma$. Local minima are called
\emph{mini-wells} and local maxima \emph{mini-saddles}. The
quantization of the classical universal form will introduce such
perturbations.

\begin{theo}\label{theo:semicl-norm-form}
  Let $I_0$ be the first Bohr-Sommerfeld invariant of $\gamma$ (see
  Subsection \ref{sec:hamiltonian-invariant}). There exists a
  neighborhood $\Omega$ of $\gamma$, $\epsilon>0$, and a Fourier
  integral operator $U: L^2(X)\to L^2(S^1)$ such that
  \begin{enumerate}
  \item $U$ is microlocally unitary from $\Omega$ to $\{(\theta,I)\in T^*S^1,|I-I_0|<\epsilon\}$.
  \item
    $Q=U P U^* = b_0 + (g_\h(\frac{\h}{i}\deriv{}{\theta}))^2 + \h
    V_\h(\theta) + R$,
    where $V_\h$ is an $\h$-dependent potential on $S^1$ with an
    asymptotic expansion
    \[
    V_\h(\theta) = V_0(\theta) + \h V_1(\theta) + \cdots,
    \]
    $g_\h\in \Cinf_0(\RM)$ is supported on an $\h$-independent
    set, with
    \[
    g_\h(I) = g_0(I) + \h g_1(I) + \cdots,
    \]
    and $g_0$ is a local diffeomorphism from a neighborhood of $I=I_0$
    to a neighborhood of $0\in\RM$. Here, $R$ is such that, for every
    $u_{\hbar}$ with
    $WF_{\hbar}(u_{\hbar})\subset \{(\theta,I)\in
    T^*S^1,|I-I_0|<\epsilon\}$, one has
    $\|Ru_{\hbar}\|=\bigO(\hbar^{\infty})\|u_{\hbar}\|$.
  \end{enumerate}
\end{theo}

The first systematic treatment of quantum mini-wells was proposed in
\cite{helffer_puits_1986}, where $P_{\hbar}=-\hbar^2\Delta+V$ is a
Schrödinger operator in several dimensions, and the potential $V$ is
Morse-Bott and minimal on a compact submanifold. Then, under a
non-degeneracy assumption on the mini-wells, one has a complete
expansion, as well as sharp decay estimates, for the lowest energy
eigenfunction of $P_{\hbar}$. This result generalizes to any
Morse-Bott principal symbol which reaches its minimum on a compact
isotropic submanifold, see~\cite{deleporte_low-energy_2017} for a
treatment in the Berezin-Toeplitz setting.

Other settings in which the principal symbol vanishes in a Morse-Bott
way include magnetic Laplacians, where the minimal set is the zero set
of the kinetic energy of the classical charged particle. In the
two-dimensional case, under the assumption that the magnetic field
does not vanish, the minimal set is symplectic, and one obtains an
effective 1D quantum Hamiltonian by viewing the minimal set as the
reduced phase space. This gives rises to spectral asymptotic to all
orders, see~\cite{raymond_geometry_2015}. In the three-dimensional
case, under the assumption of a maximal rank magnetic 2-form, and a
non-degenerate minimum for the magnetic intensity, one obtains a more
intricate reduction, which contains half-integer powers of the
semiclassical parameter $\h$, see~\cite{helffer_magnetic_2016}.

In this paper, we focus on the 1D case. Despite the critical points,
we are able to to formulate Bohr-Sommerfeld conditions (in a folded
covering) for the eigenvalues in a macroscopic window
$[\min \mathop{Sp}(P_{\hbar}),\min \mathop{Sp}(P_{\hbar})+c]$ for $c$
small, see Propositions \ref{prop:small-e} and \ref{prop:large-e}. The
invariant $I_0$ appears in the low-lying eigenvalues under a symmetry
hypothesis.

\begin{prop}\label{prop:oscill-ev}
  Let $k\geq 0$. Suppose that, in Theorem \ref{theo:semicl-norm-form}, the $k+1$ first
  terms $V_0,V_1,\ldots, V_k$ of the potential do not depend on $\theta$. Suppose
  also that $P_{\hbar}-b_0$ is elliptic at infinity. Then the
  following is true.
  \begin{enumerate}
  \item There exists $f:S^1\to \RM$ non-constant such that the first
    eigenvalue $e_{0}^{\hbar}$ of $P_{\hbar}$ satisfies:
    \[
      e_0^{\hbar}=b_0+\hbar
      V_0(0)+\hbar^2f\left(\frac{I_0}{\hbar}\text{ mod
        }\ZM\right)+\bigO(\hbar^{\max(k+2,3)}).
    \]
    
  \item Let $e_1^{\hbar}$ similarly denote the second eigenvalue of
    $P_{\hbar}$ (with multiplicity). There exists a sequence $(\hbar_j)_{j\in \NM}\to 0$ such that
    \[
      e_1^{\hbar_j}-e_0^{\hbar_j}=\bigO(\hbar^{k+2}).
      \]
  \end{enumerate}
\end{prop}

This oscillatory behaviour of the first eigenvalue was remarked in
recent work on the magnetic Laplacian \cite{helffer_thin_2019}. It is
related to the topological nature of the problem: low-energy
eigenfunctions are microsupported on a non-contractible set (here,
$\gamma$). In previous works by one of the authors~\cite{san-fn,
  san-focus} , quantum maps between open sets of non-trivial topology
were already discussed.

In the generic case where $V_0$ is a Morse function, this oscillatory
behaviour disappears at the bottom of the spectrum: because of
these subprincipal effects, eigenfunctions with energies smaller than
$b_0+\hbar \max(V_0)$ will microlocalise on a contractible set, and
one can build a quantum normal form independent of $I_0$.

This paper is organized as follows: Section \ref{sec:morse-bott-functions}
contains a classical normal form for functions admitting a
non-degenerate well on a closed loop and a reminder on the invariant $I_0$. In preparation for the quantum
normal form, Section \ref{sec:formal-perturbations} contains a
treatment of formal perturbations of the normal form above. Then, in
Section \ref{sec:semicl-norm-form} we derive a corresponding quantum
normal form, microlocally near the non-degenerate well. In Section
\ref{sec:low-energy-spectrum} we apply this quantum normal form to
obtain asymptotics of the low-lying eigenvalues.

\section{Reduction of Morse-Bott functions}
\label{sec:morse-bott-functions}

\subsection{Local symplectic normal form}
\label{sec:class-norm-form}

The Morse-Bott condition on $p$ near $\gamma$ amounts to the
following: there exist a neighbourhood $\Omega_1$ of
$\gamma$ and (non necessarily symplectic) smooth coordinates
$(t,x):\widetilde{\Omega}\to [-1,1]\times \mathbb{S}^1$, such that
\[
  p=b_0+t^2.
\]
In particular, $\gamma=\{t=0\}$ is a regular level set of the function
$t$.

By the action-angle theorem, there exists a possibly smaller open
neighborhood of $\gamma$, $\Omega_2$, equipped with symplectic
coordinates $(\theta,I)\in T^* S^1$, such that $t=g(I)$ for some
smooth function $g$, with
\begin{enumerate}
\item $\gamma = \{I=0\}$;
\item $g'(0)\neq 0$.
\end{enumerate}

Thus we obtain
\begin{prop}\label{prop:class-normal-form}
  If $p$ admits a non-degenerate well along a closed curve $\gamma$,
  then there exists ``folded action-angle'' coordinates $(\theta,I)$
  near $\gamma$ that are adapted to $p$, in the sense that
\[
p = b_0 + (g(I))^2,
\]
for some smooth function $g:(\RM,0)\to(\RM,0)$ with non-vanishing
derivative.
\end{prop}

\begin{rema}
  It follows that the set of leaves defined by $p$, \emph{i.e.} the
  space of connected components of levels sets of $p$, is a smooth
  one-dimensional manifold $\mathcal{C}$ (parameterized by $I$ or
  $\tilde I := g(I)$) , and the induced map
  $\bar p - b_0:\mathcal{C}\to\RM$ is a simple fold:
  $\tilde I \mapsto \tilde I^2$.
\end{rema}

For any $\delta\in\mathcal{C}$, and $h\in\Cinf(\hat\Omega)$, we define 
\begin{equation}
  \moy{h}_\delta := \frac{1}{2\pi}\int_0^{2\pi} h(\theta,I(\delta)) \DD \theta.
\end{equation}
Let us denote by $\phy_H^t$ the hamiltonian flow of the function $H$
at time $t$. We notice that, for all $m\in\delta$,
\begin{equation}
\label{equ:moy}
  \moy{h}_\delta = \frac{1}{2\pi}\int_0^{2\pi} (\phy_I^t)^* h (m)\DD t = 
  \frac{1}{T_\delta}\int_0^{T_\delta} (\phy_f^t)^* h (m)\DD t,
\end{equation}
where $T_\delta = \frac{2\pi}{g'(I(\delta))}$ is the period of the
Hamiltonian flow of $f$ on $\delta$. This defines a map
$m\mapsto \moy{h}_\delta \in \Cinf(\hat\Omega)$, that we denote by
$\moy{h}$.

\subsection{The first Bohr-Sommerfeld invariant}
\label{sec:hamiltonian-invariant}
Let us recall, in this context, the appearance of an invariant
associated to $\gamma$ when quantizing the symplectic change of
variables of Proposition \ref{prop:class-normal-form}.

\begin{defi}
  Suppose that either $M=\RM^2$ or $M=T^*S^1$. We define
  $I_0(\gamma)\in \RM$ as follows:
  \begin{enumerate}
  \item If $\gamma$ is contractible, it is the boundary of a close,
    compact surface $\Sigma\subset M$.
    Let $I_0=\frac{1}{2\pi}\int_{\Sigma}\omega$.
  \item If $\gamma$ is not contractible, then $M=T^*S^1$ and
    $\gamma$ is a curve with winding number $1$ with respect to
    $\theta$. For $K\in \NM$ large enough, $\gamma\cup \{\xi=-K\}$ is
    the boundary of a close, compact surface $\Sigma\subset M$. Let $I_0=\frac{1}{2\pi}(-K+\int_{\Sigma}\omega)$.
  \end{enumerate}
\end{defi}

The following proposition is well known.
\begin{prop}
  $I_0$ is a Hamiltonian invariant of $\gamma$.
\end{prop}
\begin{demo}
  In case 1, $I_0$ is clearly a symplectic invariant.

  In case 2, (and in fact the same reasoning applies to case 1 as
  well) let $\alpha = I \dd\theta$ be the canonical Liouville 1-form
  of $T^*S^1$; then we have $I_0(\gamma) = \int_\gamma \alpha$
  (because both sides of this equality vanish when $\gamma$ is
  $\{I=0\}$.)  If $X$ is a Hamiltonian vector field, then by Cartan's
  formula, $\mathcal{L}_X\alpha$ is an exact 1-form and hence acts on
  the cohomology class of $\alpha$ restricted to $\gamma$ (known as
  the Liouville class of $\gamma$). Therefore, a Hamiltonian flow
  preserves the Liouville class. On $T^* S^1$ this means that it
  preserves the integral $\int_\gamma \alpha$.

\end{demo}
\begin{rema}
  In the case 2 above, $I_0$ is \emph{not} a \emph{symplectic}
  invariant of $\gamma$; indeed any curve of the type $\{\xi=C\}$, for
  $C\in \RM$ can be sent to $\{\xi=0\}$ by the symplectic change of
  variables $(\theta,I)\mapsto (\theta,I-C)$.  However, for this
  curve, $I_0=C$.
\end{rema}

\begin{rema}
  Note that the Liouville class $I_0$ is the first Bohr-Sommerfeld
  invariant, \emph{i.e.}  the principal term in the Bohr-Sommerfeld
  cocycle defined in~\cite{san-focus} (the subprincipal terms involve
  Maslov indices and the 1-form induced by the subprincipal symbol of
  $P_\h$). In the case of Berezin-Toeplitz quantization, $I_0$ can be
  defined using parallel transport along $\gamma$ on the prequantum
  bundle \cite{charles_symbolic_2006}. In this case, $I_0$ is defined
  up to a sign and modulo $\ZM$, but the choice does not impact the
  oscillations in Proposition \ref{prop:oscill-ev} since, for Toeplitz
  quantization, $\hbar^{-1}$ takes integer values.
\end{rema}

In the rest of this section, we use Proposition
\ref{prop:class-normal-form} to build normal forms given by
Hamiltonian diffeomorphisms.

\begin{prop}\label{prop:semiglobal-ham-R2}
  Let $p:\RM^2\to \RM$ admitting a non-degenerate well along a curve
  $\gamma$.

  There exists $\epsilon>0$ and a symplectic change of variables $\sigma:\RM^2\to
  \RM^2$, equal to the identity outside of a compact set, such that,
  for all $(x,\xi)\in \RM^2$,
  \[
    |x|^2+|\xi|^2\in (2I_0-\epsilon,2I_0+\epsilon)\Rightarrow p\circ
    \sigma (x,\xi)=b_0+(g(|x|^2+|\xi|^2-2I_0))^2.
  \]
  
  \end{prop}
  \begin{demo}
    Let $r_0=\sqrt{2I_0}$.

    By Proposition \ref{prop:class-normal-form}, there exists a
    symplectic change of variables $\sigma_0$ from a neighbourhood $\Omega_0$ of
    $\gamma$ to a neighbourhood of $\{|x|^2+|\xi|^2=2I_0\}$ such that
    $p\circ \sigma_0(x,\xi)=(g(|x|^2+|\xi|^2-2I_0))^2$.

    Let $\Omega_i$ be a neighbourhood of the compact component of
    $\RM^2\setminus \Omega_0$, such that $\Omega_i$ is delimited by a connected component of a
    level set of $p$. The open set $\Omega_i$ is
    contractible; let $\phi_i:\Omega_i\to B(0,r_i)$ be a
    diffeomorphism smooth up to the boundary, where $r_i$ is fixed by
    a scaling and such that $\mathop{vol}(\Omega_i)=\pi r_i^2$.

    The map $\phi_i\circ \sigma_0^{-1}$ is a smooth diffeomorphism
    from the boundary
    $\{x^2+\xi^2=r_i^2\}$ to itself, and is the boundary value of an
    orientation-preserving, smooth map. Hence it has winding number 1
    and is smoothly isotopic to the identity. This allows us to
    correct $\phi_i$ into $\widetilde{\phi}_i$, which satisfies the
    same conditions, and such that $\phi_i\circ\sigma_0^{-1}$ is the
    identity near $\{x^2+\xi^2=r_i^2\}$.

    We apply the same strategy to a neighbourhood $\Omega_e$ of the
    infinite component of $\RM^2\setminus \Omega_0$, and obtain a
    smooth diffeomorphism $\widetilde{\phi}_e$, equal to identity
    outside a very large ball.

    Now the three smooth functions  $\widetilde{\phi}_i$, $\widetilde{\phi}_e$,
    and $\sigma_0$, coincide on the intersections of their domains of
    definition, so that gluing them yields a smooth diffeomorphism
    $\phi$ satisfying the following conditions. 
    \begin{itemize}
    \item There exists a neighbourhood $\Omega_1$ of $\gamma$ on which
      $\phi$ is a symplectomorphism and
      \[
        p=[(x,\xi)\mapsto (g(x^2+\xi^2-2I_0))^2]\circ \phi.
      \]
    \item The volume of the compact component $K_i$ of $\RM^2\setminus
      \Omega_1$ is equal to the volume of its image by $\phi$.
    \item $\phi$ is identity outside a large ball $B(0,R)$
    \item The volume of $K_e$, the infinite component of $\RM^2\setminus
      \Omega_1$ intersected with $B(0,R)$, is equal to the volume of
      its image by $\phi$.
    \end{itemize}
    It only remains to modify $\phi$ into a volume-preserving
    transformation. To this end, we will apply the Moser trick.

    On $\phi(K_i)$, consider the standard volume form $\omega_{st}$,
    and the pushed-forward volume form $\phi^*\omega_{st}$. These two
    forms coincide on a neighbourhood of the boundary and have same integral. The interpolation \[[0,1]\ni t\mapsto \omega_t=t
      \omega_{st}+(1-t)\phi^*\omega_{st}\]
    yields a family of exact symplectic forms: every (closed) $2$-form with zero integral is exact. By Moser's
    argument, there exists a diffeomorphism $\psi_i:\phi(K_i)\to
    \phi(K_i)$, sending $\phi^*\omega_{st}$ to $\omega_{st}$, and
    equal to identity near the boundary. In particular, one can
    correct $\phi$ into a symplectic change of variables on $K_i$,
    without modifying $\phi$ near the boundary of $K_i$.

    To conclude, we play the same game on $\phi(K_e)$.
  \end{demo}

  \begin{rema}
   In the previous Proposition, if a ball $B(0,c)$ lies inside
      the compact component of $\RM^2\setminus \gamma$, one can impose
      that $\sigma$ is equal to identity on $B(0,c-\epsilon)$. Indeed,
      in this case, one can prescribe that $\phi_i$ is the identity on
      $B(0,c-\epsilon/2)$, and the corrections in the rest of the
      proof preserve the fact that $\phi_i$ is the identity on
      $B(0,c-\epsilon)$.
\end{rema}

  \begin{prop}\label{prop:semiglobal-ham-T*S1}
    Let $p:T^*S^1\to \RM$ admitting a non-degenerate well along a curve
    $\gamma$. Suppose that $\gamma$ is non-contractible.

    Then there exists $\epsilon>0$ and a Hamiltonian diffeomorphism $\sigma:T^*S^1\to
    T^*S^1$, equal to the identity outside of a compact set, such
    that, for all $(x,\xi)\in T^*S^1$,
    \[
      \xi\in (I_0-\epsilon,I_0+\epsilon)\Rightarrow p\circ\sigma
        (x,\xi)=b_0+(g(\xi-I_0))^2.
        \]
  \end{prop}
  \begin{demo}
    Let $R>0$; consider the following symplectomorphism from
    $S^1\times [-2R,2R]$ to $\{(x,\xi)\in \RM^2,R\leq x^2+\xi^2\leq
    9R\}$:
    \[
      (\theta,I)\mapsto \{(\sqrt{2(I+5R/2)}\cos(\theta),\sqrt{2(I+5R/2)}\sin(\theta))\}.
    \]
    Through this symplectomorphism, we are reduced to Proposition \ref{prop:semiglobal-ham-R2}: because of the volume considerations, one can extend the
    symplectic normal form given by Proposition
    \ref{prop:class-normal-form} to a hamiltonian change of variables,
    equal to identity outside of $\{(x,\xi)\in \RM^2,R\leq x^2+\xi^2\leq
    9R\}$.
  \end{demo}

  The symplectic change of variables at the beginning of the last
  proof can be quantized; this will allow us in Section
   to quantize the normal
  form \ref{prop:class-normal-form} into a unitary operator, up to
  $\bigO(\hbar)$ error, but where $I$ is replaced with $I-I_0$. Improving
  this $\bigO(\hbar)$ error is the topic of the next section.
\section{Formal perturbations}
\label{sec:formal-perturbations}
Suppose that $p$ admits a non-degenerate well along $\gamma$, with
$p(\gamma)=b_0$, and let
\[
p_\epsilon := p + \epsilon p_1,
\]
where $p_1$ is smooth.  We consider infinitesimal Hamiltonian
deformations of $p$, \emph{i.e.} functions of the form
$\exp(\epsilon\ad{a})p = p_\epsilon + \epsilon\{a,p_\epsilon\} +
\mathcal{O}(\epsilon^2)$,
where the generator of the deformation is the smooth function $a$ (and
$\ad{a}(h):=\{a,h\} = -\ad{p}(a)$).  We have
\[
\exp(\epsilon\ad{a})p_\epsilon = p + \epsilon(p_1 + \{a,p\}) +
\mathcal{O}(\epsilon^2).
\]
This leads to the study of the cohomological equation $\{a,p\}=r$
where $r$ is given and $a$ is unknown. As in the previous section, we
let $f$ be the smooth branch of $\sqrt{p-b_0}$.
\begin{lemm}
\label{lemm:ad}
  There exists a neighborhood $\hat\Omega$ of $\gamma$ on which, for
  any $h\in\Cinf(\hat\Omega)$, the following holds.
  \begin{enumerate}
  \item $h\in \ker \ad{p}$ if and only if $h = q\circ f$ for some
    smooth function $q$.
  \item $h\in\ad{p}(\Cinf(\hat\Omega))$ if and only if
    \begin{enumerate}
    \item for all $\delta\in\mathcal{C}$, $\moy{h}_\delta=0$ and
    \item $h_{\restr \gamma} = 0$.
    \end{enumerate}
  \end{enumerate}
\end{lemm}
\begin{demo}
  \begin{enumerate}
  \item In the action-angle variables of Proposition
    \ref{prop:class-normal-form} one has
    \[
      p:(\theta,I)\mapsto b_0+(g(I))^2,
    \]
    where $g:(\RM,0)\to (\RM,0)$ is a smooth diffeomorphism.
    
    On $\Omega_2$, one has
    \[
      \{p,h\}=2g'(I)g(I)\partial_{\theta}h(\theta,I).
    \]
    In particular, $\{p,h\}=0$ if and only if $h$ depends only on $I$,
    that is, $h=q\circ f$ for some $f\in C^{\infty}(\RM,\RM)$.
  \item Let us decompose $h\in C^{\infty}(\Omega_2,\RM)$ in Fourier series in $\theta$:
    \[
      h:(\theta,I)\mapsto \sum_{k\in \ZM}h_k(I)e^{ik\theta}.
    \]
    We search for $a\in C^{\infty}(\Omega_2,\RM)$, of the form
    \[
      a:(\theta,I)\mapsto \sum_{k\in \ZM}a_k(I)e^{ik\theta}.
    \]
    such that
    \[
      \{a,p\}=h.
    \]
    One can compute
    \[
      \{a_k(I)e^{ik\theta},p\}=ik g'(I)g(I)a_k(I)e^{ik\theta}.
    \]
    The action of $\ad{p}$ is diagonal with respect to the Fourier
    series decomposition; $h$ belongs to its image if and only if
    $h_0=0$ and for
    every $k\neq 0$, $h_k$ belongs to the ideal generated by $g$, that
    is, $h_k(0)=0$. This concludes the proof.
  \end{enumerate}
\end{demo}

Let $\pi_\theta:\hat\Omega\to\gamma$ be given by
$(\theta,I)\mapsto \theta$. The space of functions that depend only on
$\theta$ is then denoted $\pi_\theta^*\Cinf(\gamma)$.

Inside $\ker \ad{p}$, let $(\ker \ad{p})_0$ denote the subspace of
functions vanishing on $\gamma$.
\begin{coro}
  There is a direct sum decomposition
  \begin{equation}\label{equ:direct}
    \Cinf(\hat\Omega) = (\ker \ad{p})_0 \oplus \ad{p}(\Cinf(\hat\Omega))
    \oplus \pi_\theta^*\Cinf(\gamma).
  \end{equation}
\end{coro}
\begin{demo}
  Let us write again $h$ as a Fourier series in $\theta$:
  \[
    h:(\theta,I)\mapsto \sum_{k\in \ZM}h_k(I)e^{ik\theta}.
  \]
  We decompose $h=h_1+h_2+h_3$, where
  \begin{align}
    (\ker \ad{p})_0\ni h_1:(\theta,I)&\mapsto h_0(I)-h_0(0)\\
    \ad{p}(C^{\infty}(\Omega_2))\ni h_2:(\theta,I)&\mapsto \sum_{k\in
                                                   \ZM^*}(h_k(I)-h_k(0))e^{ik\theta}\\
    \pi_{\theta}^*C^{\infty}(\gamma)\ni h_3:(\theta,I)&\mapsto
    \sum_{k\in \ZM}h_k(0)e^{ik\theta}.
  \end{align}
  This concludes the proof.
\end{demo}


In particular, we obtain the following:
\begin{prop}\label{prop:coh-eq}
  Given any $r\in \Cinf(\hat\Omega)$, there exists
  $a\in \Cinf(\hat\Omega)$, $q\in\Cinf(\RM, b_0)$ with $q(0)=0$, and
  $V\in \pi_\theta^*\Cinf(\gamma)$, such that
  \begin{equation}
    \{p,a\} = r - q\circ f - V.
  \end{equation}
\end{prop}

By induction, this leads to the following Birkhoff normal form.
\begin{theo}\label{theo:birkhoff}
  Let $p_\epsilon$ be a formal deformation of $p$:
  \[
  p_\epsilon \sim p + \epsilon p_1 + \epsilon^2p_2 + \cdots
  \]
  There exists a symplectic diffeomorphism $\phy_\epsilon$ in a
  neighborhood of $\gamma$, depending smoothly on $\epsilon$, such
  that
  \begin{equation}
    \phy_\epsilon^* p_\epsilon = b_0 +
    (g_{\epsilon}\circ f)^2 + \epsilon V_\epsilon + \mathcal{O}(\epsilon^\infty),
  \end{equation}
  where $g_\epsilon\in\Cinf(\RM,0)$,
  $V_\epsilon = \pi_\theta^*\widetilde V_\epsilon$ for some
  $\widetilde V_\epsilon\in\Cinf(\gamma)$; moreover both $g_\epsilon$
  and $\widetilde V_\epsilon$ (and hence $V_\epsilon$) admit an asymptotic
  expansion in integer powers of $\epsilon$ (for the $\Cinf$
  topology), and moreover $g_{\epsilon}=g+\bigO(\epsilon)$ and $g_{\epsilon}(0)=g(0)$.
 
  In other words, there exists canonical coordinates
  $(\theta,I)\in T^* S^1$ in which
  \[
  p_\epsilon(\theta,I) = b_0 + (g_{\epsilon}(I))^2 + \epsilon
  V_\epsilon(\theta) + \mathcal{O}(\epsilon^\infty).
  \]
\end{theo}
\begin{demo}
  By the classical normal form we may assume that
  $p=b_0+(g\circ f)^2$. Suppose by induction that
  \begin{equation}
    \phy_\epsilon^* p_\epsilon = b_0 +
    (g_{\epsilon}\circ f)^2 + \epsilon V_\epsilon +\epsilon^Nr,
  \end{equation}
  for some $N\geq 1$ (if $N=1$ we choose $g_\epsilon=g$ and
  $V_\epsilon=0$).

  Let $(a,q,V)$ be as in Proposition \ref{prop:coh-eq}.
  We have
  \[\exp(\epsilon^N\ad{a}) \phy_\epsilon^*p_\epsilon =
  \phy_\epsilon^*p_\epsilon +
  \epsilon^N\{a,\phy_\epsilon^*p_\epsilon\} +
  \mathcal{O}(\epsilon^{2N}).\] Hence
  \begin{equation}
    \exp(\epsilon^N\ad{a})\phy_\epsilon^*p_\epsilon =  
    b_0+(g_{\epsilon}\circ f)^2 + \epsilon V_\epsilon 
    + \epsilon^N (r + \{a, p\}) + \mathcal{O}(\epsilon^{N+1}),
  \end{equation}
  with
  \[
    r+\{a,p\}=q\circ f + V
  \]
  where $q(0)=0$.
  
 Hence
\begin{align}
\label{equ:birkhoff}
\exp(\epsilon^N\ad{a})\phy_\epsilon^*p_\epsilon &=
b_0+(g_{\epsilon}\circ f)^2 + \epsilon^Nq\circ f + \epsilon (V_\epsilon +
                                                  \epsilon^{N-1}V) + \mathcal{O}(\epsilon^{N+1})\\
  &=
b_0+\left[\left(g_{\epsilon}+\epsilon^N\frac 12 q\right)\circ f\right]^2 + \epsilon (V_\epsilon +
                                                  \epsilon^{N-1}V) + \mathcal{O}(\epsilon^{N+1})
\end{align}
Finally if we assumed that $\phy_\epsilon$ was the time-one flow of a
Hamiltonian $a_\epsilon$, we see that the left-hand side
of~\eqref{equ:birkhoff} is the flow of the Hamiltonian
$a_\epsilon + \epsilon^N a$ modulo $\mathcal{O}(\epsilon^{N+1})$. This
proves the induction step.
\end{demo}

\section{Semiclassical normal form}
\label{sec:semicl-norm-form}

\subsection{Quantum maps}
\label{sec:quantum-maps}

In order to quantize the results of Section
\ref{sec:morse-bott-functions}, we need a proper notion of quantum map
corresponding to a symplectic change of variables.

In the whole of this section, to simplify notation, we will use the
subscript $\hbar$ to denote that an object depends on a parameter $\hbar$ belonging to a punctured neighbourhood of zero within a closed subset of $\RM^+$.

\begin{defi}
  Let $(M^1,\sigma^1,H^1_{\hbar},Op^1_{\hbar})$ and
  $(M^2,\sigma^2,H^2_{\hbar},Op^2_{\hbar})$ be two quantization
  procedures: for $i=1,2$, $(M^i,\sigma^i)$ are symplectic manifolds, $H^i_{\hbar}$
  are ($\hbar$-dependent) Hilbert spaces and
  $Op^i_{\hbar}:C^{\infty}_c(M,\CM)\to B(H^i_{\hbar})$ realise
  formal deformations of the Poisson algebras
  $C^{\infty}_c(M^i,\CM)$. The functors $Op^i_{\hbar}$ yield natural
  notions of $\hbar$-wave front set for families of elements of $H^i_{\hbar}$.

  A {\bfseries quantum map} consists of data
  $(U_{\hbar},\Omega_1,\Omega_2,\sigma)$, where $\Omega_1,\Omega_2$
  are respectively open subsets of $M_1$ and $M_2$,
  $\sigma:\Omega_1\to \Omega_2$ is a smooth and proper symplectomorphism, and
  $U_{\hbar}:H^1_{\hbar}\to H^2_{\hbar}$ satisfies the following
  properties:
  \begin{enumerate}
  \item For every $K\subset \subset \Omega_1$, for every $u_{\hbar}\in H^1$ with
    $\|u_{\hbar}\|_{H^1}=1$ and $WF_{\hbar}(u_{\hbar})\subset K$, one has
    \[
      \|U_{\hbar}u_{\hbar}\|_{H^2}=1+\bigO_K(\hbar^{\infty}).
    \]
  \item For every $K\subset \subset \Omega_2$, for every $v_{\hbar}\in H^2$ with
    $\|v_{\hbar}\|_{H^2}=1$ and $WF_{\hbar}(v_{\hbar})\subset K$, one has
    \[
      \|U_{\hbar}^*v_{\hbar}\|_{H^1}=1+\bigO_K(\hbar^{\infty}).
    \]
  \item For every $a\in C^{\infty}_c(M_2,\RM)$, there exists a sequence
    $(b_k)_{k\geq 0}\in \left[C^{\infty}_c(M_1,\RM)\right]^{\NM_0}$,
    such that $b_0=a\circ \sigma$ and for all $K\subset \subset
    \Omega_2$, for every $v_{\hbar}\in H^2$ with
    $\|v_{\hbar}\|_{H^2}=1$ and $WF_{\hbar}(v_{\hbar})\subset K$, one
    has
    \[
      U_{\hbar}Op_{\hbar}^2(a)U_{\hbar}^*v=\sum_{k=0}^{\infty}\hbar^{-k}Op_{\hbar}^1(b_k)v+O(\hbar^{\infty}).
      \]
    \end{enumerate}
    A linear operator $U_{\hbar}$ satisfying conditions 1 and 2 above
    will be called a microlocal unitary transform.
\end{defi}
A broad class of examples of quantum maps is given by the Egorov
Theorem (see \cite{zworski_semiclassical_2012}). Indeed, if
$(M^1,\omega^1)=(M^2,\omega^2)=T^*X$ where $X$ is a smooth, compact
manifold, if $Op$ is the Weyl quantization, and if $\sigma$ is a
global Hamiltonian transformation (corresponding to a time-dependent
Hamiltonian $H(t)$ for $t\in [0,1]$), then one can construct
$U_{\hbar}$ as follows: for $u_0\in L^2(X)$, $U_{\hbar}u_0$ is the
solution at time $t=1$ of the differential equation $i\hbar
\partial_tu(t)=Op^W_{\hbar}(H(t))u(t)$ with initial value
$u(0)=u_0$. This procedure also works in more general quantization contexts.

In this section, we will use two particular quantum maps from
$T^*S^1$ to $\RM^2$, which we define now.

\begin{defi}
  Let $\Omega_1=S^1\times \RM^+_*$ and $\Omega_2=\RM^2\setminus \{0\}$
  be respectively open sets of $T^*S^1$ and $\RM^2$. Let
  $\sigma:\Omega_1\to \Omega_2$ be defined as
  \[
(\theta,I)\mapsto (\sqrt{2I}\cos(\theta),\sqrt{2I}\sin(\theta)).
\]

  For $\hbar>0$ and $k\in \NM_0$, let $\phi_{k,\hbar}\in L^2(\RM)$ denote the
  $k$-th Hermite eigenfunction of the $\hbar$-harmonic oscillator,
  defined by the following induction relation:
  \begin{align}
    \phi_{0,\hbar}&:x\mapsto \frac{1}{\sqrt{2\pi
        \hbar}}e^{-\frac{x^2}{2\hbar}}\\
    \phi_{k+1,\hbar}&=\frac{1}{\hbar\sqrt{2(k+1)}}(-\hbar \partial + x)\phi_{k,\hbar}\qquad \qquad
                      \text{ for }k\geq 0.
  \end{align}
  
  The toric quantum map
  $(\mathcal{T}_{\hbar},\Omega_1,\Omega_2,\sigma)$ is defined by its
  action on the Fourier basis as
  \[
    \mathcal{T}_{\hbar}(\theta\mapsto
    e^{ik\theta})=\begin{cases}\phi_{k,\hbar}&\text{ if } k\geq 0\\
      0&\text{ if } k<0.
    \end{cases}
  \]
\end{defi}
\begin{prop}
  The toric quantum map is indeed a quantum map.
\end{prop}
\begin{demo}
  By definition, one has, for $k\geq 0$,
  \[
    \mathcal{T}_{\hbar}^*(-\hbar \partial +
      x)\mathcal{T}_{\hbar}(\theta\mapsto
    e^{ik\theta})=(\theta\mapsto \sqrt{2}\hbar\sqrt{k+1}e^{i(k+1)\theta}).
  \]
  In other terms, if $Op^1_{\hbar}$ denotes left quantization \cite{zworski_semiclassical_2012},
  \[
    \mathcal{T}_{\hbar}^*(-\hbar \partial +
    x)\mathcal{T}_{\hbar}=Op^1_{\hbar}(\sqrt{2I}\mathds{1}_{I\geq 0}e^{i\theta}).
  \]
Weyl quantization and left quantization are equivalent for smooth
symbols. Hence, there exists $(b_k)_{k\in \NM_{>0}}\in
\left[C^{\infty}(S^1\times \RM^*_+,\RM)\right]^{\NM_{>0}}$ such that,
for all $K\subset \subset S^1\times \RM^*_+$, for all $u_{\hbar}\in
L^2(S^1)$ normalised with $WF_{\hbar}(u_{\hbar})\subset K$, one has
\[
  \mathcal{T}_{\hbar}^*(-\hbar \partial +
    x)\mathcal{T}_{\hbar}u=Op^W_{\hbar}\left(\sqrt{2I}\mathds{1}_{I\geq
        0}e^{i\theta}+\sum_{k=1}^{+\infty}\hbar^{-k}b_k(\theta,I)\right)u+\bigO_K(\hbar^{\infty}).
  \]

  Taking the symmetric and antisymmetric part yields, with the same hypotheses,
  \begin{align}
    \mathcal{T}_{\hbar}^*Op_{\hbar}^W(x)\mathcal{T}_{\hbar}u&=Op_{W}^{\hbar}\left(\sqrt{2I}\cos(\theta)+\sum_{k=1}^{+\infty}\hbar^{-k}\mathop{Re}(b_k)(\theta,I)\right)u\\
    \mathcal{T}_{\hbar}^*Op_{\hbar}^W(\xi)\mathcal{T}_{\hbar}u&=Op_{W}^{\hbar}\left(\sqrt{2I}\sin(\theta)+\sum_{k=1}^{+\infty}\hbar^{-k}\mathop{Im}(b_k)(\theta,I)\right)u.
  \end{align}
  Then, by the Weyl calculus, one can determine
  $\mathcal{T}_{\hbar}^*Op_{\hbar}(P(x,\xi))\mathcal{T}_{\hbar}$ for any
  polynomial $P$, and eventually of any smooth function.
\end{demo}

\begin{defi}
  Let $(x_0,\xi_0)\in \RM^2$ and let $r<\pi$. Let
  \begin{align}
    \Omega_1&=\{(\theta,I)\in S^1\times \RM, \mathop{dist}(\theta+2\pi
           \ZM,x_0)^2+(I-\xi_0)^2<r\}\\
    \Omega_2&=\{(x,\xi)\in \RM^2,(x-x_0)^2+(\xi-\xi_0)^2<r\}
 .
  \end{align}
  Let $\sigma_{x_0,\xi_0,r}:\Omega_1\to \Omega_2$ be defined by
  $(\theta,I)\mapsto (x_{\theta},I)$ where $x_{\theta}\in
  \theta+2\pi\ZM$ and
  $\mathop{dist}(x_{\theta},x_0)=\mathop{dist}(\theta+2\pi\ZM,x_0)$.
  Let $\chi:\RM\mapsto [0,1]$ be a smooth function equal to $1$ on a
  neighbourhood of $[-r,r]$ and to 0 on a neighbourhood of
  $\RM\setminus [-\pi,\pi]$.

  We then define $W_{x_0,\xi_0,r}:L^2(S^1)\to L^2(\RM)$ as follows:
  for $u\in L^2(S^1)$,
  \[
    W_{x_0,\xi_0,r}u:x\mapsto
    \chi(x-x_0)Op^W_{\hbar}(\mathds{1}_{(\theta,I)\in
      \Omega_1})u(x\text{ mod } 2\pi\ZM),
  \]
  and we define the unrolling quantum map as $(W_{x_0,\xi_0,r},\Omega_1,\Omega_2,\sigma)$.
  \end{defi}
The unrolling quantum map is a quantum map by definition of $Op^W$ on
$T^*S^1$.

\subsection{Quantization of the normal form}
\label{sec:quant-norm-form}

From now on, $M=T^*X$, with $X=\RM$ or $X=S^1$; our semiclassical
analysis will be concerned with Weyl quantization. The results can be
transported to other geometrical settings (manifolds with
asymptotically conic or hyperbolic ends,
Berezin-Toeplitz quantization of compact manifolds, ...) as long as one has a good notion
of ellipticity at infinity and a microlocal equivalence with Weyl
quantization, and provided that one can make sense of the invariant
$I_0$ above. One should note, however, that the Morse condition of
Section \ref{sec:morse-case} is not invariant under a change of quantization.

Let $P$ be a semiclassical
pseudo-differential operator on $X$ with a classical symbol in a
standard class,
\[
p_\h(x,\xi) = p_0(x,\xi) + \h p_1(x,\xi) + \cdots
\]
We assume that the principal symbol $p_0$ admits a non-degenerate well
on a loop $\gamma$.

We are now ready to prove Theorem \ref{theo:semicl-norm-form} 

\begin{demo}
  One proceeds as in Theorem \ref{theo:birkhoff}. The starting point
  is a quantization $U_0$ of the symplectic normal form given by Proposition
  \ref{prop:class-normal-form}. 

  In our setting, there are three possible topological situations for
  $\gamma$, and we give the three corresponding constructions of
  $U_0$.
  \begin{enumerate}
  \item If $M=\RM^2$, then $\gamma$ is contractible and one can apply
    Proposition \ref{prop:semiglobal-ham-R2}. Let $H$ be a
    (time-dependent) Hamiltonian satisfying the conditions of
    Proposition \ref{prop:semiglobal-ham-R2} (in particular, $H$ is
    constant near infinity, so it belongs to the symbol class
    $S_0$). We let $\exp(-i\hbar^{-1}\hat{H})$ be the corresponding quantum
    evolution. We now let, for $\epsilon>0$ small enough,
    \[
      U_0=\mathcal{T}_{\hbar}^*\exp(i\hbar^{-1}\hat{H}).
    \]
    
  \item If $M=T^*S^1$ and $\gamma$ is contractible, we let $\Sigma$ be
    the compact connected component of $M\setminus \gamma$, and we let
    $(B((\theta_i,\xi_i),r_i))_{i\in \mathcal{I}}$ be a finite
    covering of a contractible neighbourhood of $\Sigma$ by disks of
    radius $<\pi$, and $(\chi_i)_{i\in \mathcal{I}}$ be an associated
    partition of unity. We then let $(x_i)_{i\in \mathcal{I}}$ be a family
    of real numbers such that $[x_i]=\theta_i$ and
    $(B((x_i,\xi_i),r_i))_{i\in I}$ is a covering of a connected
    preimage $\hat{\Sigma}$ of $\Sigma$ by the rolling map. Then, we define
      \[
        W=\sum_{i\in
          \mathcal{I}}W_{x_i,\theta_i,r_i}Op_W^{\hbar}(\chi_i).
      \]
      Near $\hat{\Sigma}$, one can apply Proposition
      \ref{prop:semiglobal-ham-R2} as in the previous case, and we let
      \[
        U_0=\mathcal{T}_{\hbar}^*\exp(-i\hbar^{-1}\hat{H})W.\]
      
    \item If $M=T^*S^1$ and $\gamma$ is not contractible, then we apply
      Proposition \ref{prop:semiglobal-ham-T*S1}; if $H$ is a
      (time-dependent) Hamiltonian satisfying Proposition
      \ref{prop:semiglobal-ham-T*S1}, then we let
      \[
        U_0=\exp(-i\hbar^{-1}\hat{H}).
        \]
  \end{enumerate}
In all cases, by the Egorov theorem, there exists a sequence
$(q_k)_{k\geq 1}$ of symbols such that, for all $u$ microlocalised
in a neighbourhood of $\{\xi=I_0\}$, one has
\[Q_0u:=U_0PU_0^*u=b_0u+\left(g_0\left(\frac{\hbar}{i}\frac{\partial}{\partial
    \theta}\right)\right)^2u+\sum_{k=1}^{\infty}\hbar^{-k}Op^W_{\hbar}(q_k)u+\bigO(\hbar^{\infty}).\]

It remains to correct $U_0$ by induction, in order to get an
$\bigO(\hbar^{\infty})$ remainder. To this end, we proceed exactly as in
Theorem \ref{theo:birkhoff}, replacing $\exp(\hbar^N\ad{a})$ with
$\exp(i\hbar^{N-1}Op_W^{\hbar}(a))$ which acts the same way up to a
next-order error.

\end{demo}

\section{Low-energy spectrum under global ellipticity}
\label{sec:low-energy-spectrum}

If $\gamma$ is a global minimum for $p$, then from
Theorem \ref{theo:semicl-norm-form} one can hope to describe the
spectrum of $P_{\hbar}$ at low energies. This section is devoted to
the spectral study of $Q_{\hbar}$ under two different assumptions.

\begin{enumerate}
\item Case where $V_0$ is constant. When $\h$ varies, the
  eigenvalues are located on smooth branches (parabolas) and the
  smallest eigenvalue regularly ``jumps'' from one branch to the other
  (See Figure~\ref{fig:spectre}). In the case of Schrödinger operators
  with a strong magnetic field, this oscillatory effect is known as
  ``Little-Parks'', see Figure 1 in
  \cite{kachmar_counterexample_2019} and \cite{fournais2015lack}.
\begin{figure}[h]\label{fig:spectre}
  \centering
  \includegraphics[width=0.8\linewidth]{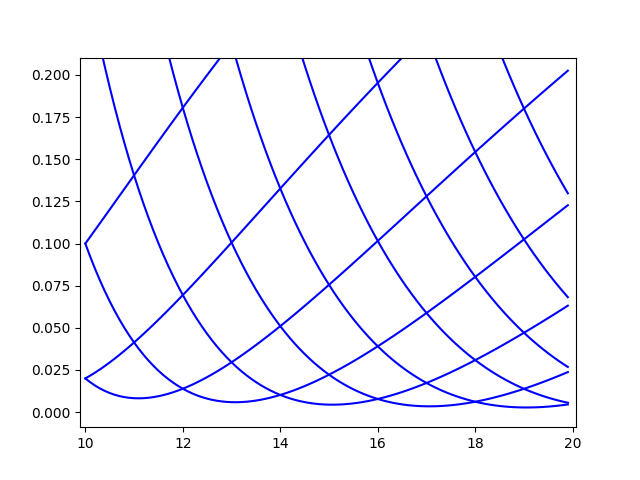}
  \caption{Small eigenvalues for the operator
    $Op_W^{\hbar}((x^2+y^2-1)^2)$ acting on $L^2(\RM)$, as a function
    of $1/\hbar$. The first eigenvalue jumps branches when $1/\hbar$
    is a multiple of $
    \frac{1}{I_0}=2$.}
\end{figure}

\item Generic subprincipal symbol. Then we can reduce to a
  Schrödinger-like operator with Morse potential $V$, but after a $\sqrt{\h}$
  zoom in the variable $I$. We consider the two following interesting cases.
  \begin{enumerate}
  \item local mimina of the potential: we get ``mini-wells'';
  \item local maxima: we can describe the concentration on hyperbolic
    trajectories.
  \end{enumerate}
\end{enumerate}
Before studying these assumptions, we recall that the microlocal
knowledge of $P$ near $\gamma$ is sufficient to treat low-energy
eigenvalues if the symbol is elliptic at infinity.

\subsection{Microlocal confinement}
From now on, we make the following assumption:

\begin{assu}\label{ass:p-ell}
  The curve $\gamma$ is a global minimum for $p$. Moreover, $p-b_0$ is elliptic at
  infinity in a scattering symbol class $S^{m,\ell}$ with $m\geq 0$
  and $\ell \geq 0$. That is,
  \[
    \forall (j,k)\in \NM^2,\exists C_{jk}\in \RM, \forall (x,\xi)\in T^*M,|\partial_x^j\partial_{\xi}^k
    p(x,\xi)|<C_{jk}(1+|x|)^{\ell-j}(1+|\xi|)^{m-k}.
    \]
\end{assu}

\begin{prop}\label{prop:loc-1}
  Suppose Assumption \ref{ass:p-ell} holds. Then there exists $E_0>b_0$ and $h_0>0$ such that, for
  any $\hbar<h_0$ and any eigenpair $(u,E)$ of $P$ with $E<E_0$ and
  $\|u\|_{L^2(X)}=1$, one has
  $\|Uu\|_{L^2(\SM^1)}=1+\bigO(\hbar^{\infty})$ and
  \[
    \|QUu-Uu\|_{L^2(\SM^1)}=\bigO(\hbar^{\infty}).
    \]
  \end{prop}
  \begin{demo}
    Let $E_0>p(\gamma)$ be such that
    \[
      \{p\leq E_0\}\subset\subset \Omega\qquad\qquad
      \phi_0^{-1}(\{p\leq E_0\})\subset\subset \{|I-I_0|\leq \eta\}.\]
    By standard elliptic estimates (see Appendix E in
    \cite{dyatlov_mathematical_2019}, for instance), $u$ is $\bigO(\hbar^{\infty})$ outside
    $\{p\leq E_0\}$, so that $\|Uu\|_{L^2}=1+\bigO(\hbar^{\infty})$ by
    item 1 in
    Theorem \ref{theo:semicl-norm-form}; moreover the characteristic 
    of $U$ is the graph of $\phi_0^{-1}$, so that
    $WF_{\hbar}(Uu)\subset \{|I-I_0|\leq \eta\}$. Thus, one can
    apply item 2 in Theorem \ref{theo:semicl-norm-form}.    
  \end{demo}
The $S_{\rho,1-\rho}$-calculus for $\rho<\frac 12$ then leads to the
following, more precise localisation estimate.
  \begin{prop}\label{prop:loc-2}
    Suppose Assumption \ref{ass:p-ell} holds. Let $\delta>0$ and $\delta'>0$. For every $\hbar^{1-\delta}\leq E \leq E_0$,
    where $E_0$ is as in Proposition \ref{prop:loc-2}, for every unit
    eigenfunction $v$ of $Q$ with eigenvalue $E$, $\hat{v}$ is
    $\bigO_{\delta,\delta'}(h^{\infty})$ on $\{|I-I_0|\geq \hbar^{\frac{1-\delta-\delta'}{2}}\}$.
  \end{prop}
  Here, for $v\in L^2(\mathbb{S}^1)$, $\hat{v}$ is the
  semiclassical discrete Fourier transform of $v$, which we view as an element of
  $\ell^2(\hbar \mathbb{Z})$.

\subsection{Case with a symmetry}
\label{sec:case-with-symmetry}

In this section we suppose that $V_0$ is constant, and we prove Proposition
\ref{prop:oscill-ev}.

We first give a proof in the simpler case when one has also $V_1$ constant.
\begin{prop}
  Suppose that Assumption \ref{ass:p-ell} holds and that $V_{\hbar}$
  does not depend on $\theta$ modulo $\hbar^N$. Let $E_0$ be
  as in Proposition \ref{prop:loc-1}.
  The eigenvalues of $P$ in the window $(-\infty,b_0+E)$ are given up to a
  uniform $\bigO(\hbar^{N+1})$ error by
  \[
    \{b_0+\hbar V_{\hbar}(0)+g_{\hbar}(\hbar k)^2\cap [0,c),\,k\in \ZM
   \}.
    \]
  \end{prop}
  \begin{demo}
    From Proposition \ref{prop:loc-1}, the eigenvalues of $P$ in the
    window above are exactly given by eigenvalues of $Q$ in the same
    window, up to an $\bigO(\hbar^{\infty})$ error. Reciprocally,
    since low-energy eigenfunctions of $Q$ are themselves
    microlocalised in $\{|\xi-I_0|<\epsilon\}$, small eigenvalues of
    $Q$ are $\bigO(\hbar^{\infty})$-close to the spectrum of $P$.
    
    If $V_{\hbar}$ does
    not depend on $\theta$ up to some error, then $Q$ is a Fourier
    multiplier (up to this error), whose
    eigenvalues are the values at $I\in \hbar \mathbb{Z}$.
  \end{demo}

  \begin{prop}
    Suppose that Assumption \ref{ass:p-ell} holds and that $V_0$ does
    not depend on $\theta$.

    Then the first eigenvalue of $P_{\hbar}$ is given, up to
    $\bigO(\hbar^3)$, by $b_0+\hbar (g_1(I_0)+V_0)+\hbar^2f(I_0\hbar^{-1})$,
    where $f$ is a non-constant, 1-periodic function.
  \end{prop}
  \begin{proof}
    For all $k\in \ZM$, let
    \[
      \lambda_k=(k-I_0\hbar^{-1})g_1'(I_0)+(k-I_0\hbar^{-1})^{2}g_0'(I_0).
    \]
    Let us also write a Fourier decomposition of $V_1$ as
    \[
      V_1:\theta\mapsto \sum_{l\in \ZM}v_le^{il\theta}.
    \]
    Then, by the ellipticity assumption, the first eigenvalue of
    $P_{\hbar}$ coincides, modulo $\bigO(\hbar^3)$, with the first
    eigenvalue of
    \[
      b_0+\hbar(V_0+g_1(I_0))+\hbar^2A
    \]
    where $A$ is the following operator on $\ell^2(\ZM)$:
    \[
      \forall (k,l)\in \ZM^2,A_{k,l}=\begin{cases}\lambda_k+v_0&\text{
          if } k=l\\
        v_{l-k}&\text{ if } k\neq l.\end{cases}
    \]
    The spectrum of the operator $A$ is 1-periodic as a function of
    $\sigma =
    I_0\hbar^{-1}$. Indeed, \[\lambda_k(\sigma)=\lambda_{k+1}(\sigma+1).\]

    In particular, the first eigenvalue of $P_{\hbar}$ has the
    requested form, but it remains to prove that $f$ is not constant.

    To this end, observe that $A$ has compact resolvent and analytic
    dependence on $\sigma$, so that if its first eigenvalue is
    constant, the corresponding eigenspace $E_0$ is also constant.

    However, we observe that
    $\partial^2_{\sigma}A=g_0'(I_0)^2\text{Id}$, with $g_0'(I_0)\neq
    0$. In particular, since $E_0$ does not depend on $\sigma$,
    $\partial^2_{\sigma}A|_{E_0}=g_0'(I_0)^2\text{Id}$, so that the
    first eigenvalue cannot be constant. This concludes the proof.
  \end{proof}

  \begin{rema}
    Since $g_0^2$ reaches a non-degenerate minimum at $I_0$, the first
    eigenvalue of $P$ is, in this case,
    \[
      b_0+\hbar g_1(I_0)+\hbar(\hbar k_{\hbar}-I_0)g_1'(I_0)+(\hbar
      k_{\hbar}-I_0)^2g_0'(I_0)^2+\bigO(\hbar^3),
    \]
    where
    \[
      k_{\hbar}=\left\lfloor \frac{I_0}{\hbar}-\frac 12 g_1'(I_0)-\frac 12\right\rfloor,
      \]
      for typical values of $\hbar$ (unless $\frac{I_0}{\hbar}-\frac 12
      g_1'(I_0)-\frac 12$ is $\hbar$-close to an integer, in which
      case it might be $k_{\hbar}+1$ or $k_{\hbar}-1$). In particular,
      this proves Proposition \ref{prop:oscill-ev}.

      The function $V_0$ is the pseudodifferential equivalent of the
      ``Melin value'' $\mu$ introduced and studied in
      \cite{deleporte_low-energy_2017}. In particular, if the
      subprincipal symbol $p_1$ of the original operator is
      identically zero, then so is $V_0$. However, the term
      $V_1$ is, in general, non-zero.
  \end{rema}

  \begin{ex}
    Let $S\in \frac 12 \NM_{>0}$. Consider the normalized spin operator
    \[
      S_z^2=\frac{1}{4(S+1)^2}\begin{pmatrix}
        (-S)^2&&&&\\
        &(-S+1)^2&&&\\
        &&\ddots&&\\
        &&&(S-1)^2&\\
        &&&&S^2
      \end{pmatrix}.
    \]
    This operator is the Berezin-Toeplitz quantization of the symbol
    $(x,y,z)\mapsto z^2-\hbar$ on $\SM^2$, where the semiclassical parameter
    is $\hbar=\frac{1}{2S}$. This symbol vanishes on the equator in a
    Morse-Bott way; here $I_0=\frac 12$. In this rotational invariant
    case, one has $V=0$.

    Even though $\hbar$ is a discrete parameter, the oscillation
    phenomenon of Figure \ref{fig:spectre} is also found here: for
    integer values of $S$, the lowest eigenvalue of $S_z^2$ is $0$;
    whereas for half-integer values of $S$ it is $\frac{1}{8(S+1)^2}$.

    Spin operators are models for magnetism in solids. In some
    contexts, the behaviour of a spin system is expected to strongly
    depend on whether the spin is integer or half-integer (Haldane
    conjecture). These effects may be related to the model case above.
  \end{ex}
\subsection{Morse case}
\label{sec:morse-case}
In this section we make Assumption 2. We give Bohr-Sommerfeld
quantization rules in two overlapping regimes: the first one consists
of energies
smaller than $b_0+C\hbar$ for any fixed $C>0$. The second consists
of energies in the window $[b_0+C\hbar, b_0+c]$ for $C>0$ large enough
and $c>0$ small enough. Propositions \ref{prop:small-e} and
\ref{prop:large-e} yield together the spectrum of $P_0$ up to energies $b_0+c$.
\subsubsection{Small energies}
\label{sec:small-energies}

\begin{prop}\label{prop:small-e}
  Let the following operators act on $L^2(\mathbb{S}^1)$:
  
  \begin{align}
    H_0&=g_0'(I_0)^2\left(\frac{\sqrt{\hbar}}{i}\frac{\partial}{\partial
        \theta}\right)^2+V_0(\theta)\\
    H_1&=2g_0'(I_0)\left[g_1(I_0)+g_0'(I_0)\left(\frac{I_0}{\hbar}-\left\lfloor\frac{I_0}{\hbar}\right\rfloor\right)\right]\frac{\sqrt{\hbar}}{i}\frac{\partial}{\partial \theta}.
  \end{align}
  
Let $C>0$ and $\epsilon>0$. Then there exists $C_1>0$ such that the spectrum of
$P_{\hbar}$, in the interval $[b_0,b_0+C\hbar]$, is the spectrum of
$H_0+\sqrt{\hbar}H_1$ in the interval $[0,2C]$, composed by the affine function $\lambda\mapsto
b_0+\hbar \lambda$, and up to an error uniformly bounded by
$C_1\hbar^{2-\epsilon}$.
\end{prop}
\begin{rema}
  The operator $H_0+\sqrt{\hbar}H_1$ is the quantization of a
  symbol on $L^2(\mathbb{S}^1)$, with semiclassical parameter $\sqrt{\hbar}$; $H_0$
  corresponds to the principal part and $H_1$ to the subprincipal
  part. The spectrum of this operator, on fixed intervals, can be
  described by Bohr-Sommerfeld rules if $V$ is Morse: we refer to \cite{duistermaat_oscillatory_1974} for
  the regular case,
 \cite{colin_de_verdiere_spectre_1980} for the elliptic case (a), and \cite{colin_de_verdiere_conditions_1998} for the
 hyperbolic case (b).

  In particular, away from the critical values of $V_0$, for instance
  on $[\max V_0+c,C]$, the principal
  symbol of $H_0$ is regular and consists of two connected
  components. On each of these components, the Bohr-Sommerfeld
  rule yield $\bigO(\hbar)$-quasimodes for $H_0+\sqrt{\hbar}H_1$,
  whose associated eigenvalues are separated by $\epsilon\sqrt{\hbar}$ for
  $\epsilon$ small enough depending on $c$. Eigenmodes corresponding
  to different components are microlocalized on disjoint regions of
  phase space (respectively $\{\xi>c\}$ and $\{\xi<-c\}$ so that they
  do not interact up to $\bigO(\hbar^{\infty})$. In conclusion, for
  $\hbar$ small enough, by a
  perturbative argument, one can construct
  $\bigO(\hbar^{\infty})$-quasimodes for $Q$ in this spectral region,
  yielding $\bigO(\hbar^{\infty})$-quasimodes for $P$ in the region
  $[b_0+\hbar(\max V_0+c),b_0+\hbar C]$.
\end{rema}
\begin{demo}
  First, by Proposition \ref{prop:loc-1} we are reduced to the study
  of the spectrum $Q$ in the same interval $[b_0,b_0+C\hbar]$.

  By Proposition \ref{prop:loc-2}, any eigenfunction $v$ of $Q$ in
  this interval is localised in frequency in $\{|\xi-I_0|\leq
  C\hbar^{\frac 12-\epsilon}\}$ for all $\epsilon>0$. In particular,
  if the Taylor expansion of $g_0$ and $g_1$ around $I_0$ are
  \begin{align}
    g_0(I)&=g_0'(I_0)(I-I_0)+\frac{g_0''(I_0)}{2}(I-I_0)^2+\bigO((I-I_0)^3)\\
    g_1(I)&=g_1(I_0)+\bigO(I-I_0),
  \end{align}
  then
  \begin{multline}
    \left[g_0\left(\frac{\hbar}{i}\frac{\partial}{\partial
          \theta}\right)+\hbar
        g_1\left(\frac{\hbar}{i}\frac{\partial}{\partial
            \theta}\right)\right]^2v\\=\left[g_0'(I_0)\left(\frac{\hbar}{i}\frac{\partial}{\partial
            \theta}-I_0\right)+\frac{g_0''(I_0)}{2}\left(\frac{\hbar}{i}\frac{\partial}{\partial
            \theta}-I_0\right)^2+\hbar g_1(I_0)+\bigO(\hbar^{\frac32 -3\epsilon})\right]^2v\\
                                    =\hbar\left[g_0'(I_0)^2D_{\hbar}^2+\sqrt{\hbar}g_0'(I_0)\left(2g_1(I_0)+g_0''(I_0)D_{\hbar}^2\right)D_{\hbar}+\bigO(\hbar^{1-3\epsilon})\right]v
\end{multline}
where we introduce
\[
  D_{\hbar}=\frac{\sqrt{\hbar}}{i}\frac{\partial}{\partial
    \theta}-\frac{I_0}{\sqrt{\hbar}}.\]

Notice that, the unitary conjugation
on $L^2(\mathbb{S}^1)$ given by
multiplication by
\[x\mapsto \exp\left(i\left\lfloor\frac{I_0}{\hbar}\right\rfloor\right)\]
amounts to replacing $D_{\hbar}$ with                                  \[
  \widetilde{D_{\hbar}}=\frac{\sqrt{\hbar}}{i}\frac{\partial}{\partial
    \theta}-\sqrt{h}\{I_0\}_{\hbar}\]
where
\[
  \{I_0\}_{\hbar}=\frac{I_0}{\hbar}-\left\lfloor\frac{I_0}{\hbar}\right\rfloor=O_{\hbar\to
    0}(1).
\]

In conclusion, the eigenvalues of
$Q$ in the interval
$[b_0,b_0+C\hbar]$ are given, up to
$\bigO(\hbar^{2-3\epsilon})$, by the
eigenvalues of
\[
  \left[g_0'(I_0)^2\widetilde{D_{\hbar}}^2+V_0(\theta)\right]
  +\hbar^{\frac
    12}g_0'(I_0)\left[2g_1(I_0)+g_0''(I_0)\widetilde{D}_{\hbar}^2\right]\widetilde{D}_{\hbar}
  \]
in the window $[0,C]$, pushed by the map $\lambda\mapsto
b_0+\hbar\lambda$. This concludes the proof.
\end{demo}
\begin{rema}
  If $V_0$ is Morse, the smallest eigenvalue of $P$ admits an
  expansion in powers of $\sqrt{\hbar}$
  \cite{deleporte_low-energy_2017}. The oscillations in Figure
  \ref{fig:spectre}, of order $\bigO(\hbar^2)$, are destroyed by the
  perturbation induced by $V_0$, which at this scale is of order
  $\hbar^{3/2}$.

  This fact stresses out again the topological nature of the invariant
  $I_0$. If $V_0$ is Morse, the lowest-energy eigenfunctions of $P$ will
  microlocalise near the minimal points of $V_0$, so that a quantum
  normal form only needs to be built in a neighbourhood of these
  points, instead of in a whole neighbourhood of $\gamma$.
\end{rema}

\subsubsection{Large energies}
\label{sec:large-eigenvalues}

It remains to study the spectrum of $Q$ in the window
$[b_0+C\hbar,b_0+c_1]$ for $C$ large enough.

To this end, let $E\in [2C\hbar,c_1]$; we will determine the
eigenvalues of $Q$ in the window $[b_0+\frac{E}{2},b_0+2E]$ up to an
error $\bigO(\hbar^2)$ uniform in $E$.

Since $g_0(I_0)=0$ and $g_0\in C^{\infty}([I_0-c,I_0+c],\RM)$, there exists
$\widetilde{g_0}\in C^{\infty}([-c,c],\RM)$ such that
\[
  g_0(I)=(I-I_0)\widetilde{g_0}(I).
\]
In particular, the following function belongs to
$C^{\infty}([-c,c]\times [-c,c],\RM)$:
\[
  f:(x,y)\mapsto
  \frac{1}{x}g_0(xy+I_0)=y\widetilde{g_0}(xy+I_0).
\]
In particular, $f(0,y)=(g_0'(I_0)y)$.

The function
\[
  h_0^{E,t}:(\theta,\eta)\mapsto
  f^2(\sqrt{E},\eta)+tV_0(\theta),
\]
is then a continuous deformation of $h_0^{0,0}=f^2(0,\eta)$, whose
Hamiltonian trajectories are circles.

We also let
\[
  h_1^{E}:(\theta,\eta)\mapsto 2f(\sqrt{E},\eta)g_1(\eta\sqrt{E}+I_0).
  \]

We let $c_1>0,c_2>0$ be such that, for $0\leq E \leq c_1$ and $0\leq t
\leq c_2$, the hamiltonian trajectories of $h_0^{E,t}$ of energies in
the window $\left[\frac 13, 3\right]$ are nondegenerate circles.

Now
\[
  \frac{1}{E}(Q_{\hbar}-b_0)=\frac{1}{E}g_0\left(\frac{\hbar}{i}\frac{\partial}{\partial
      \theta}\right)^2+2\frac{\hbar}{E}g_0\left(\frac{\hbar}{i}\frac{\partial}{\partial
        \theta}\right)g_1\left(\frac{\hbar}{i}\frac{\partial}{\partial
          \theta}\right)+\frac{\hbar}{E}V_{0}(\theta)+O\left(\frac{\hbar^2}{E}\right)
      \]
      where
      \[
        \frac{1}{E}g_0\left(\frac{\hbar}{i}\frac{\partial}{\partial
            \theta}\right)^2+\frac{h}{E}V_0(\theta)=Op_{W}^{\frac{\hbar}{\sqrt{E}}}\left(h_0^{E,\frac{h}{E}}\left(\theta,\eta-\frac{I_0}{\sqrt{E}}\right)\right)\]
      and
      \[
        2\frac{\hbar}{E}g_0\left(\frac{\hbar}{i}\frac{\partial}{\partial
        \theta}\right)g_1\left(\frac{\hbar}{i}\frac{\partial}{\partial
        \theta}\right)=\frac{\hbar}{\sqrt{E}}Op_W^{\frac{\hbar}{\sqrt{E}}}\left(h_1^E\left(\theta,\eta-\frac{I_0}{\sqrt{E}}\right)\right).\]
As previously, after unitary conjugation with $x\mapsto
\exp\left(-ix\left\lfloor \frac{I_0}{\hbar}\right\rfloor\right)$, one can
  replace $\frac{I_0}{\sqrt{E}}$ with $\frac{\hbar}{\sqrt{E}}\{I_0\}_{\hbar}$.

\begin{prop}\label{prop:large-e}
  Let $E\in \left[\frac{1}{c_2}\hbar,c_1\right]$. The eigenvalues of
  $P_{\hbar}$ in the window $\left[b_0+\frac E2,b_0+2E\right]$ are given by the
  eigenvalues of
  \[
    Op_W^{\frac{\hbar}{\sqrt{E}}}\left(h_0^{E,\frac{\hbar}{E}}\right)+\frac{\hbar}{\sqrt{E}}Op_W^{\frac{\hbar}{\sqrt{E}}}\left(h_1^E\right)\]
    in the window $\left[\frac 12,2\right]$, by the transformation
    \[
      \lambda \mapsto b_0+\frac{\lambda}{E},
    \]
    up to an error $\bigO(\hbar^2)$, uniform in $E$.
  \end{prop}
 
  By definition of $c_2$, the Hamiltonian trajectories of
  $h_0^{E,\frac{\hbar}{E}}$ are non-degenerate circles, so that the
  eigenvalues and eigenfunctions of the model operator are given by
  the Bohr-Sommerfeld rules.

Again, the error $\bigO(\hbar^2)$ is very small compared to the spectral gap of the
  model operator in each branch, which is $\hbar\sqrt{E}$, as long as
  $\hbar$ is small enough. Hence, in practical cases one can determine
  $\bigO(\hbar^{\infty})$-quasimodes for $P$.

\section{Acknowledgements}
\label{sec:acknowledgements}

This work emerged from a discussion at a CNRS GDR ``DYNQUA'' in Lille, and
we gratefully acknowledge the importance of such GDR meetings.

Part of the paper is based upon work supported by the National Science
Foundation under Grant No. DMS-1440140 while A. Deleporte was in
residence at the Mathematical Sciences Research Institute in Berkeley,
California, during the Fall 2019 semester.

\nocite{duistermaat1972fourier,le2014singular,san2006systemes}

\bibliographystyle{abbrv}%
\bibliography{Circular-Well,Circular-Well-suppl}

\begin{thebibliography}{10}

\bibitem{charles_symbolic_2006}
L.~Charles.
\newblock Symbolic calculus for {{Toeplitz}} operators with half-form.
\newblock {\em Journal of Symplectic Geometry}, 4(2):171--198, 2006.

\bibitem{colin_de_verdiere_spectre_1980}
Y.~Colin De~Verdi{\`e}re.
\newblock Spectre conjoint d'op{\'e}rateurs pseudo-diff{\'e}rentiels qui
  commutent.
\newblock {\em Mathematische Zeitschrift}, 171(1):51--73, 1980.

\bibitem{colin_de_verdiere_conditions_1998}
Y.~Colin De~Verdi{\`e}re and B.~Parisse.
\newblock Conditions de {{Bohr}}-{{Sommerfeld}} singuli{\`e}res.
\newblock {\em Preprint Institut Fourier}, 432, 1998.

\bibitem{deleporte_low-energy_2017}
A.~Deleporte.
\newblock Low-energy spectrum of {{Toeplitz}} operators with a miniwell.
\newblock {\em Communications in Mathematical Physics}, (to appear), 2019.

\bibitem{duistermaat_oscillatory_1974}
J.~J. Duistermaat.
\newblock Oscillatory integrals, {{Lagrange}} immersions and unfolding of
  singularities.
\newblock {\em Communications on Pure and Applied Mathematics}, 27(2):207--281,
  1974.

\bibitem{duistermaat1972fourier}
J.~J. Duistermaat, L.~H{\"o}rmander, et~al.
\newblock Fourier integral operators. ii.
\newblock {\em Acta mathematica}, 128:183--269, 1972.

\bibitem{dyatlov_mathematical_2019}
S.~Dyatlov and M.~Zworski.
\newblock {\em Mathematical Theory of Scattering Resonances}, volume 200.
\newblock {American Mathematical Soc.}, 2019.

\bibitem{fournais2015lack}
S.~Fournais and M.~P. Sundqvist.
\newblock {Lack of diamagnetism and the Little--Parks effect}.
\newblock {\em Communications in Mathematical Physics}, 337(1):191--224, 2015.

\bibitem{helffer_thin_2019}
B.~Helffer and A.~Kachmar.
\newblock Thin domain limit and counterexamples to strong diamagnetism.
\newblock {\em arXiv:1905.06152 [math-ph]}, May 2019.

\bibitem{helffer_magnetic_2016}
B.~Helffer, Y.~Kordyukov, N.~Raymond, and S.~V{\~u}~Ng{\d o}c.
\newblock Magnetic {{Wells}} in {{Dimension Three}}.
\newblock {\em Analysis \& PDE}, 9(7):1575--1608, Nov. 2016.

\bibitem{helffer_puits_1986}
B.~Helffer and J.~Sj{\"o}strand.
\newblock Puits multiples en limite semi-classique {{V}} : {{{\'E}tude}} des
  minipuits.
\newblock {\em Current topics in partial differential equations}, pages
  133--186, 1986.

\bibitem{kachmar_counterexample_2019}
A.~Kachmar and M.~P. Sundqvist.
\newblock Counterexample to strong diamagnetism for the magnetic {{Robin
  Laplacian}}.
\newblock {\em arXiv:1910.12499 [math-ph]}, Nov. 2019.

\bibitem{le2014singular}
Y.~Le~Floch.
\newblock {Singular Bohr-Sommerfeld conditions for 1D Toeplitz operators:
  elliptic case}.
\newblock {\em Communications in Partial Differential Equations},
  39(2):213--243, 2014.

\bibitem{raymond_geometry_2015}
N.~Raymond and S.~V{\~u}~Ng{\d o}c.
\newblock Geometry and spectrum in {{2D}} magnetic wells.
\newblock {\em Annales de l'Institut Fourier}, 65:137--169, 2015.

\bibitem{san-focus}
S.~V{\~u}~Ng{\d o}c.
\newblock Bohr-{S}ommerfeld conditions for integrable systems with critical
  manifolds of focus-focus type.
\newblock {\em Comm. Pure Appl. Math.}, 53(2):143--217, 2000.

\bibitem{san-fn}
S.~V{\~u}~Ng{\d o}c.
\newblock Formes normales semi-classiques des syst{\`e}mes compl{\`e}tement
  int{\'e}grables au voisinage d'un point critique de l'application moment.
\newblock {\em Asymptotic Analysis}, 24(3,4):319--342, 2000.

\bibitem{san2006systemes}
S.~Vũ~Ng{\d{o}}c.
\newblock Syst{\`e}mes int{\'e}grales semi-classiques: Du local au global.
\newblock {\em Panoramas et synth{\`e}ses-Soci{\'e}t{\'e} math{\'e}matique de
  France}, (22):1--151, 2006.

\bibitem{zworski_semiclassical_2012}
M.~Zworski.
\newblock {\em Semiclassical Analysis}, volume 138.
\newblock {American Mathematical Soc.}, 2012.

\end{thebibliography}
\end{document}